\documentclass[12pt]{amsart}
\usepackage{enumerate}

\newtheorem{thm}{Theorem}
\newtheorem{lem}[thm]{Lemma}
\newtheorem{prop}[thm]{Proposition}
\newtheorem*{defn}{Definition}

\newtheorem*{theorm}{Theorem}

\begin{document}

\title{Affine coordinates and finiteness}
\author{Julien Giol}
\date{February 2010}

\address{C\'egep du Vieux-Montr\'eal, Montr\'eal, Canada}
\email{julien.giol@gmail.com}

\thanks{This research was conducted while I was a postdoc at Texas A\&M University. 
I am grateful to this institution, and to Gilles Pisier in particular, for their support during this period.}











\begin{abstract}
We show that a von Neumann algebra is finite if and only if its Grassmannians are small in a certain sense related to the atlas of affine coordinates.
\end{abstract}

\maketitle

\section{Introduction}

A set is usually called finite if it can be put in bijection with a set of the form $\{1,\ldots,n\}$ for some natural number $n$.
Yet this is not the definition operator algebraists look at when defining the dichotomy finite/infinite for von Neumann algebras.
Instead, they mimick a definition due to Dedekind which requires the Axiom of Choice to be proven equivalent to the classical one within Zermelo-Fraenkel set theory.

A set is called Dedekind-infinite if it can be put in bijection with a proper subset of itself.
Then a set $X$ is Dedekind-finite if it is not Dedekind-infinite, which amounts to the following:
$$ Y\subseteq X \quad\mbox{and}\quad Y\simeq X\quad\Rightarrow\quad Y=X$$
where $Y\simeq X$ denotes the equivalence relation: there is a bijection from $Y$ onto $X$.


A fundamental quality of von Neumann algebras is that they contain plenty of projections, i.e. self-adjoint idempotents.
More specifically, denoting
$$P(M)=\{p\in M\,|\,p^2=p=p\sp*\}$$
the set of projections in a von Neumann algebra $M$, we have that the span of $P(M)$ is dense in $M$ with respect to the operator norm topology.
To exploit this abundance of projections is the key idea that led Murray and von Neumann to the celebrated classification of factors, i.e. von Neumann algebras $M$ such that $M'\cap M=\mathbb{C}1$, into types $\rm{I}_n, \rm{I}_\infty, \rm{II}_1, \rm{II}_\infty, \rm{III}$.

The set $P(M)$ is always equipped with the following relations:
\begin{itemize}
\item \emph{order:} $p\leq q$ if the range of $p$, $\mbox{Im}\,p$, is contained in $\mbox{Im}\,q$. 
Note that this is equivalent to the algebraic condition $pq=qp=p$.
\item \emph{equivalence:} $p\sim q$ if there exists $u$ in $M$ such that $p=u\sp*u$ and $q=uu\sp*$. 
Such an element $u$ realizes an isometry from $\mbox{Im}\,p$ onto $\mbox{Im}\,q$. 
The so-called Murray-von Neumann equivalence relation $p\sim q$ means that the closed subspaces $\mbox{Im}\,p$ and $\mbox{Im}\,q$ are isometric within $M$. 
\item \emph{homotopy:}  $p$ and $q$ are called homotopic if there exists a projection-valued path connecting $p$ and $q$ within $P(M)$.
In a von Neumann algebra, two projections $p,q$ are homotopic if and only if they are \emph{unitarily equivalent}, i.e. if there exists a unitary $u$ such that $q=upu\sp*$.
\end{itemize}


\begin{defn}
A projection $q$ in $P(M)$ is called finite if the following holds
$$ p\leq q \quad\mbox{and}\quad p\sim q\quad\Rightarrow\quad p=q$$
for every projection $q$ in $P(M)$.
The von Neumann algebra $M$ is called finite if the unit $1$ is a finite projection.
\end{defn}

Remak that $M$ is finite if and only if every projection $p$ in $P(M)$ is finite.
Observe also that $M$ is finite if and only if every isometry is onto.
Finally, note that every von Neumann algebra contained in a finite one is finite.

If homotopy always implies Murray-von Neumann equivalence, these two notions coincide in a finite von Neumann algebra.
This is actually characteristic for finiteness.

Examples of finite von Neumann algebras include finite-dimensional $C\sp*$-algebras and group von Neumann algebras.
The algebra $B(H)$ of bounded linear operators on the infinite-dimensional separable Hilbert space $H$ is a typical example of infinite von Neumann algebra.

Most of the early classification results on von Neumann algebras involve mostly set-theoretical approaches and arguments. 
The purpose of this paper is to show that one can recover the same dichotomy finite/infinite from a geometric perspective.

When identifying as usual a closed subspace $F$ in $H$ with the projection onto $F$, we obtain a natural way to generalize the notion of Grassmannian.
First, the Grassmannian $G(k,n)$ of $k$-dimensional subspaces in $\mathbb{C}^n$ is identified with the projections of rank $k$ in $M_n(\mathbb{C})$.
Thus the Grassmannians $G(k,n)$ can be seen as the building blocks of $P(M_n(\mathbb{C}))$.
More generally, the set of projections $P(M)$ splits into connected components which correspond to the equivalence classes of homotopy or, equivalently, of unitary equivalence.
It follows that the unitary orbits
$$G_p(M)=\{upu\sp*\,|\,u\in U(M)\},$$
where $u$ runs over $U(M)=\{u\in M\,|\, u\sp*u=uu\sp*=1\}$ the unitary group of $M$, constitute an infinite-dimensional analogue of the Grassmannians $G(k,n)$.

We show in Section \ref{affinecoordinates} that the classical \cite{griffithsharris} atlas of affine coordinates on $G(k,n)$ has an analogue in the context of $G_p(M)$.
The affine coordinates map is given essentially by a formula used by Kovarik \cite{kovarik77} to construct piecewise affine idempotent-valued paths.
The inverse of this map yields a so-called rational parametrization of the open unit ball centered at $p$
$$U_p=\{q\in P(M)\,|\,\|q-p\|<1\}.$$

\begin{theorm}
A von Neumann algebra $M$ is finite if and only if $U_p\cap U_q \neq\emptyset$ for every pair of homotopic projections in $P(M)$.
\end{theorm}

One way to look at this result is to say that a von Neumann algebra is finite if its Grassmannians are small in a certain sense.
On the other hand, it is infinite if it contains two disjoint open unit balls in a certain Grassmannian.
Actually, as we notice at the end of Section \ref{infinitecase}, an infinite von Neumann algebra contains a Grassmannian with infinitely many pairwise disjoint unit balls.

The paper is organized as follows.
The Section \ref{preliminaries} is devoted to some preliminary results concerning idempotents.
All of these can be found elsewhere, but we choose to include them here for the reader's convenience.
We show how to define affine coordinates on Grassmannians in von Neumann algebras in Section \ref{affinecoordinates}.
Again, our discussion of the $G(k,n)$ is essentially the one in \cite{griffithsharris}.
But we feel that it will help the reader to find this account here.
The proof of the Theorem is divided into two parts.
The finite case is treated in Section \ref{finitecase}, the infinite one in Section \ref{infinitecase}.

The results contained in this paper constitute a natural continuation of the following list of references \cite{kovarik77,zemanek79,aupetit81,esterle83,tremon85,holmes92,lauzontreil04,giol05,giol07}.

\section{Preliminary results}\label{preliminaries}

An idempotent $p$ in $B(H)$ is an element such that $p^2=p$.
It is characterized by the decomposition of $H$ into the topological direct sum of its range and its nullspace $H=\mbox{Im}\,p\oplus\mbox{Ker}\,p$.
We denote $p\sp\perp$ the idempotent $p\sp\perp=1-p$.
Note that this operation amounts to exchanging the roles of the range $\mbox{Im}\,p$ and the nullspace $\mbox{Ker}\,p$.

\begin{lem}\label{imker}
Let $p,q$ be idempotents in $B(H)$.
The following assertions are equivalent:
\begin{enumerate}[(i)]
\item ${Im}\,p=\mbox{Im}\,q$  (respectively $\mbox{Ker}\,p=\mbox{Ker}\,q$)
\item $p\sp\perp q=q\sp\perp p=0$ (respectively  $pq\sp\perp=qp\sp\perp=0$)
\end{enumerate}
\end{lem}

\begin{proof}
Since $\mbox{Ker }\,p\sp\perp=\mbox{Im }\,p$, we see that the identity $p\sp\perp q=0$ is equivalent to the containment $\mbox{Im}\,q\subseteq \mbox{Im}\,p$.
Likewise $q\sp\perp p=0$ is equivalent to $\mbox{Im}\,p\subseteq \mbox{Im}\,q$ and the first equivalence follows.
The other equivalence can be established by replacing $p$ and $q$ with $p\sp\perp$ and $q\sp\perp$.
\end{proof}

\begin{lem}\label{kerpq}
Let $p,q$ be idempotents in $B(H)$.
We have:
\begin{enumerate}[(i)]
\item $\mbox{Ker}\,(p+q-1)=\mbox{Im}\,p\cap\mbox{Ker}\,q\oplus\mbox{Ker}\,p\cap\mbox{Im}\,q$
\item $\mbox{Ker}\,(p-q)=\mbox{Im}\,p\cap\mbox{Im}\,q\oplus\mbox{Ker}\,p\cap\mbox{Ker}\,q$
\item $\mbox{Ker}\,(pq-qp)=\mbox{Ker}\,(p+q-1)\oplus \mbox{Ker}\,(p-q)$
\end{enumerate}
\end{lem}

\begin{proof}
First note that each of the subspaces $\mbox{Im}\,p\cap\mbox{Ker}\,q$ and $\mbox{Ker}\,p\cap\mbox{Im}\,q$ is contained in $\mbox{Ker}\,(p+q-1)$.
Since for instance  $\mbox{Im}\,p$ and $\mbox{Ker}\,p$ are in direct sum, we get $\mbox{Ker}\,(p+q-1)\supseteq\mbox{Im}\,p\cap\mbox{Ker}\,q\oplus\mbox{Ker}\,p\cap\mbox{Im}\,q$.
Now take $x$ in $\mbox{Ker}\,(p+q-1)$.
Note that $px=q\sp\perp x$ and that $p\sp\perp x=qx$.
Writing $x=px+p\sp\perp x$ we see that $px=q\sp\perp x$ belongs to $\mbox{Im}\,p\cap\mbox{Ker}\,q$ and that $p\sp\perp x=qx$ belongs to $\mbox{Ker}\,p\cap\mbox{Im}\,q$.
This proves that  $\mbox{Ker}\,(p+q-1)$ is contained in $\mbox{Im}\,p\cap\mbox{Ker}\,q\oplus\mbox{Ker}\,p\cap\mbox{Im}\,q$, which completes the proof of \emph{(i)}.
It suffices to change $q$ into $q\sp\perp$ to deduce \emph{(ii)} from \emph{(i)}.
Since $p$ and $q$ commute on the subspace $\mbox{Ker}\,(pq-qp)$, it is possible to diagonalize them simultaneously, which yields a decomposition into the direct sum of the four subspaces $\mbox{Im}\,p\cap\mbox{Im}\,q$, $\mbox{Ker}\,p\cap\mbox{Ker}\,q$, $\mbox{Im}\,p\cap\mbox{Ker}\,q$, and $\mbox{Ker}\,p\cap\mbox{Im}\,q$.
Together with \emph{(i)} and \emph{(ii)}, this proves \emph{(iii)}.
\end{proof}

\begin{lem}[Kovarik's formula]\label{kovarik}
Let $p,q$ be idempotents in $B(H)$.
If $p+q-1$ is invertible then there exists an idempotent $r$ such that $\mbox{Im}\,r=\mbox{Im}\,p$ and $\mbox{Ker}\,r=\mbox{Ker}\,q$.
It is given by the formula
$$r=p(p+q-1)^{-2}q$$
\end{lem}

\begin{proof}
First a routine verification shows that $\omega=(p+q-1)^2$ and its inverse commute with $p$ and with $q$ as soon as $p$ and $q$ are idempotents.
More precisely, we have $p\omega=\omega p=pqp$ and $q\omega=\omega q=qpq$.
Then the formula defines an element $r$ such that $r^2=(p\omega^{-1}q)(p\omega^{-1}q)=(pqp\omega^{-1})\omega^{-1}q=p\omega^{-1}q=r$, i.e. $r$ is an idempotent.
Finally, it is easy to check that $\mbox{Im}\,r=\mbox{Im}\,p$ and $\mbox{Ker}\,r=\mbox{Ker}\,q$ thanks to the algebraic characterizations of these identities exhibited in Lemma \ref{imker}.
\end{proof}

\begin{lem}\label{conversekovarik}
Let $p,q$ be idempotents in $B(H)$.
The element $p+q-1$ is invertible if and only if there exist $r_1$ and $r_2$ idempotents such that:
\begin{enumerate}
\item  $\mbox{Im}\,r_1=\mbox{Im}\,p$ and $\mbox{Ker}\,r_1=\mbox{Ker}\,q$
\item  $\mbox{Ker}\,r_2=\mbox{Ker}\,p$ and $\mbox{Im}\,r_2=\mbox{Im}\,q$.
\end{enumerate}
\end{lem}

\begin{proof}
If $p+q-1$ is invertible, then the formula in Lemma \ref{kovarik} allows us to define $r_1$ and $r_2$ such that the conditions above be fulfilled.
Conversely, assume that such idempotents exist.
Then a routine verification shows that $r_1+r_2-1$ is the inverse of $p+q-1$, using identities of Lemma \ref{imker}.
\end{proof}

\begin{lem}\label{corokovarikbis}
Let $p,q$ be idempotents in $B(H)$.
\begin{enumerate}[(i)]
\item If $p+q-1$ is invertible, then there exists an idempotent-valued path connecting $p$ and $q$.
\item If $p$ and $q$ are projections and if $\|p-q\|<1$, then $p$ and $q$ are homotopic.
\end{enumerate}
\end{lem}

\begin{proof}
If $p+q-1$ is invertible, then the formula of Lemma \ref{kovarik} defines an idempotent such that both segments $[p,r]$ and $[r,q]$ are contained in the set of idempotents in $B(H)$.
This follows from straightforward computations using Lemma \ref{imker}.
Thus we have proved \emph{(i)}.
If $p$ and $q$ are projections such that $\|p-q\|<1$, it follows from the formula $(p+q-1)^2=1-(p-q)^2$ that $p+q-1$ is invertible.
Hence by \emph{(i)} we can take $f(t)$ a parametrization of an idempotent-valued path connecting $p$ and $q$.
Now since $p$ and $q$ are projections, the formula
$$g(t):=f(t)(f(t)+f(t)\sp*-1)^{-2}f(t)\sp*$$
defines, thanks to Lemma \ref{kovarik} again, a projection-valued path connecting $p$ and $q$.
Note that the invertibility of $f(t)+f(t)\sp*-1$ follows from Lemma \ref{conversekovarik}.
So $p$ and $q$ are homotopic and the proof of \emph{(ii)} is complete.
\end{proof}

\begin{lem}\label{corokovarik}
Let $p,q$ be projections in $B(H)$.
The following assertions are equivalent:
\begin{enumerate}[(i)]
\item $p+q-1$ is invertible 
\item $H=\mbox{Im}\,p\oplus\mbox{Ker}\,q$
\item $\|p-q\|<1$
\end{enumerate}
\end{lem}

\begin{proof}
If $p+q-1$ is invertible, then Lemma \ref{kovarik} yields an idempotent $r$ such that $H=\mbox{Im}\,r\oplus\mbox{Ker}\,r=\mbox{Im}\,p\oplus\mbox{Ker}\,q$, so \emph{(i)} implies \emph{(ii)}.
Now if $H=\mbox{Im}\,p\oplus\mbox{Ker}\,q$, we can define an idempotent $r$ such that  $\mbox{Im}\,r=\mbox{Im}\,p$ and $\mbox{Ker}\,r=\mbox{Ker}\,q$.
Next we observe that $r\sp*$ is an idempotent such that $\mbox{Im}\,r\sp*=\mbox{Im}\,q$ and $\mbox{Ker}\,r\sp*=\mbox{Ker}\,p$.
Therefore, by Lemma \ref{conversekovarik}, we have that $p+q-1$ is invertible.
So \emph{(i)} and \emph{(ii)} are equivalent.
The equivalence between \emph{(i)} and \emph{(iii)} follows from the identity $(p+q-1)^2=1-(p-q)^2$ by functional calculus.
\end{proof}

\section{Affine coordinates}\label{affinecoordinates}

\subsection{The classical case $G(k,n)$}

To a family $\{v_1,\ldots,v_k\}$ of $k$ vectors in $\mathbb{C}^n$ we can associate $L_v$ the $n\times k$ matrix whose columns are the vectors $v_j$'s.
Then we note that the span of the $v_j$'s, namely the range of $L_v$, is a $k$-dimensional subspace of $\mathbb{C}^n$ if and only if the matrix $L_v$ is left-invertible.
Moreover, we see that Im $L_v$ = Im $L_w$ if and only if there exists an invertible element $\sigma$ in $GL(k)$ such that $L_w=L_v\sigma$.
This observation yields an identification between unordered bases of $k$-dimensional subspaces and left invertible $n\times k$-matrices, whose set we will denote $M_{n,k}\sp*$, mod out by $GL(k)$.
In other terms, we see $G(k,n)$, the set of $k$-dimensional subspaces in $\mathbb{C}^n$, via the identification 
$$G(k,n)\equiv M_{n,k}\sp*/GL(k).$$

Next observe that a matrix $L$ is in $M_{n,k}\sp*$ if and only if there is a subset of indices $I=\{i_1<\ldots<i_k\}$ in $\{1,\ldots,n\}$ such that the corresponding $I\times\{1,\ldots,k\}$ minor is invertible.
Moreover, in this case, all the matrices in $L \cdot GL(k)$ share this property with respect to the same minor extraction.
Let us denote $U_I$ the corresponding class in $M_{n,k}\sp*/GL(k)$.

Since for each matrix in $M_{n,k}\sp*$, there is an invertible $k\times k$ minor, we see that 
$$G(k,n)=\bigcup_{I=\{i_1<\ldots<i_k\}} U_I$$
where the union runs over all $k$-tuples in $\{1,\ldots,n\}$.

Let us now pick a $k$-dimensional subspace $L\cdot GL(k)$ in $U_I$.
One can see that there is a unique representative in this class such that the $I\times\{1,\ldots,k\}$ minor be equal to the identity matrix.
Then the remaining entries constitute an $(n-k)\times k$ matrix which is called the set of affine coordinates of the corresponding $k$-dimensional subspace.

Going back to the matrix $L_v$, note that the projection onto the span of $L_v$ is given by the matrix $q=L_v(L_v\sp*L_v)^{-1}L_v\sp*$. 
For example, consider the case where $I=\{1,\ldots,k\}$.
Then
$$
L_v=
\left(\begin{matrix}
1_k\cr
A
\end{matrix}\right)
\qquad\mbox{and}\qquad
q=
\left(\begin{matrix}
(1_k+A\sp*A)^{-1} &  (1_k+A\sp*A)^{-1}A\sp*  \cr
A(1_k+A\sp*A)^{-1} & A(1_k+A\sp*A)^{-1}A\sp*
\end{matrix}\right)
$$
where $1_k$ denotes the identity $k\times k$ matrix, and where $A$ is the $(n-k)\times k$ matrix corresponding to the affine coordinates.

Further, if we denote $p_I$ the projection onto $F_I$, the $k$-dimensional subspace generated by the vectors of the canonical basis corresponding to the indices in $I$, we have
$$
p_I=
\left(\begin{matrix}
1_k &  0_{k,n-k} \cr
0_{n-k,k} & 0_{n-k}
\end{matrix}\right)
$$
and a straightforward computation shows
$$
q(q+p_I-1)^{-2}p_I=
\left(\begin{matrix}
1 & 0\cr
A &0
\end{matrix}\right).
$$
In particular, we see that the formula of Lemma \ref{kovarik} helps recover the affine coordinates $A$ from the projection $q$.

Finally, note that the subspaces corresponding to the classes of $U_I$ are described as follows:
$$U_I \equiv  \{ F\in G(k,n) \;|\; F\cap F_I\sp\perp =\{0\} \}.$$
For dimension reasons, we see that if $F$ has dimension $k$, then $F\cap F_I\sp\perp =\{0\}$ if and only if $\mathbb{C}^n=F\oplus F_I\sp\perp$.
Given Lemma \ref{corokovarik}, it follows that $U_I$ corresponds to the open unit ball centered at $p_I$ in $P(M_n(\mathbb{C}))$.

\subsection{Generalization in a von Neumann algebra}

Let $M$ be a von Neumann algebra in $B(H)$ now. 
Instead of fixing a subset of indices $I$, we fix a projection $p$ in $P(M)$.
Then the equivalent of $G(k,n)$ becomes the connected component of $p$ in $P(M)$, which we denote $G_p(M)$.
The open unit ball 
$$U_p:=\{q\in P(M)\;|\; \|q-p\|<1\}$$
replaces $U_I$.
And the following result tells us what the affine coordinates map becomes in this setting.

\begin{thm}
Let $p$ be a projection in a von Neumann algebra $M$.
The map
$$\phi_p:q\longmapsto p(p+q-1)^{-2}q-p$$
realizes a homeomorphism from $U_p$ onto $p\sp\perp Mp$.
The inverse is given by the map
$$x\longmapsto
\left(\begin{matrix}
(p+x\sp*x)^{-1} &  (p+x\sp*x)^{-1}x\sp*  \cr
x(p+x\sp*x)^{-1} & x(p+x\sp*x)^{-1}x\sp*
\end{matrix}\right).
$$
\end{thm}

\begin{proof}
It suffices to consider Theorem 1.2 in \cite{giol07} for idempotents and to restrict both maps to the self-adjoint part of their domains.
Is is clear from Lemma \ref{corokovarik} that the open set $U_p=\{q | p+q-1 \mbox{ invertible}\}$ in \cite{giol07} becomes the open unit ball centered at $p$.
Now the range of $\phi_p$ in \cite{giol07} is $\Omega_p=\{h | 2p-1+h \mbox{ invertible}\}$ in the tangent space $T_p=pMp\sp\perp\oplus p\sp\perp Mp$.
Writing 
$$h=
\left(\begin{matrix}
0 &  x\sp*  \cr
x & 0
\end{matrix}\right)
\qquad \mbox{ with } x\in p\sp\perp Mp,$$
we see that 
$$(2p-1+h)^2=
\left(\begin{matrix}
p &  x\sp*  \cr
x & -p\sp\perp
\end{matrix}\right)^2
=
\left(\begin{matrix}
p+x\sp*x &  0  \cr
0 & p\sp\perp+xx\sp*
\end{matrix}\right)
$$
is always invertible when $h$ is self-adjoint.
Hence $\Omega_p$ becomes the whole self-adjoint part of the tangent space $T_p$, which can clearly be identified with $p\sp\perp Mp$.
After this identification, the map $\phi_p$ of \cite{giol07} reads $\phi_p(q)=p(p+q-1)^{-2}q-p$ on every projection $q$ in $U_p$.
Also, the rational parametrization takes the form announced above.
\end{proof}

For instance, in the case of $M=M_2(\mathbb{C})$ and for
$$p=
\left(\begin{matrix}
1 &  0  \cr
0 & 0
\end{matrix}\right)
$$
the parametrization above yields
$$z\longmapsto 
\left(\begin{matrix}
\frac{1}{1+|z|^2} &  \frac{\bar{z}}{1+|z|^2}  \cr
\frac{z}{1+|z|^2} & \frac{|z|^2}{1+|z|^2}
\end{matrix}\right).
$$
  

\section{Finite case}\label{finitecase}

\begin{prop}\label{genericcase}
Let $p,q$ be projections in $P(M)$.
If $p+q-1$ is injective, then $U_p\cap U_q \neq\emptyset$.
\end{prop}

\begin{proof}
Since $p+q-1$ is injective, we deduce from Lemma \ref{kerpq} the decomposition 
$$H=\mbox{Im}\,p\cap \mbox{Im}\,q\oplus\mbox{Ker}\,p\cap\mbox{Ker}\,q\oplus \mbox{Ker}\,(pq-qp)\sp\perp.$$
This is stable under $p$ and $q$, hence it suffices to establish the result for each of the three possible restrictions. 
The first two cases are trivial, so we can assume without loss of generality that $H=\mbox{Ker}\,(pq-qp)\sp\perp$, i.e. $\mbox{Ker}\,(pq-qp)=\{0\}$.
Then by a result of Halmos \cite{halmos69}, we can further assume that $p$ and $q$ are written as follows:
$$
p=
\left(\begin{matrix}
1 &  0  \cr
0 & 0
\end{matrix}\right)
\qquad\mbox{and}\qquad 
q=
\left(\begin{matrix}
c^2 &  cs  \cr
cs & s^2
\end{matrix}\right)
$$
where $c,s$ are positive injective contractions such that $c^2+s^2=1$.
Note that this decomposition remains in $M$, since the only tools used by Halmos in his proof are the polar decomposition and some continuous functional calculus.
Then it suffices to set 
$$r:=
\left(\begin{matrix}
\frac{1+c}{2} &  \frac{s}{2}  \cr
\frac{s}{2} & \frac{1-c}{2}
\end{matrix}\right)
$$
to obtain the desired projection in $U_p\cap U_q$.
Indeed, it is easy to check that this formula defines a projection.
Then set
$$
\tau:=
\left(\begin{matrix}
c &  s  \cr
s & -c
\end{matrix}\right).
$$
Some straightforward computations show that $\tau$ is an involution, i.e. $\tau=\tau\sp*=\tau^{-1}$, and that we have $\tau p \tau =q$ and $\tau r=r\tau=r$.
Hence $p+r-1$ is invertible, i.e. $r\in U_p$, if and only $\tau(p+r-1)\tau=q+r-1$ is invertible, i.e. $r\in U_q$.
Finally, we can easily compute
$$
(p-r)^2=
\left(\begin{matrix}
\frac{1-c}{2} &  0 \cr
0 & \frac{1-c}{2}
\end{matrix}\right).
$$
This shows that $\|p-r\|\leq 1/\sqrt{2}$, so $r\in U_p$, and the proof is complete.
\end{proof}

\begin{prop}\label{pperpcase}
Let $p$ be a projection in $P(M)$.
If $p$ and $p\sp\perp$ are homotopic, then $U_p\cap U_{p\sp\perp} \neq\emptyset$.
\end{prop}

\begin{proof}
Take $u$ a partial isometry in $M$ such that $p=uu\sp*$ and $p\sp\perp=u\sp*u$.
Without loss of generality, we can assume that $u=pup\sp\perp$, so that
$$u^2=(u\sp*)^2=up=p\sp\perp u=pu\sp*=u\sp*p\sp\perp=0.$$
Now if we put 
$$r:=\frac{1+u+u\sp*}{2},$$ 
some simple computations show that we obtain a projection such that
$$(r-p)^2=\frac{(p\sp\perp-p+u+u\sp*)^2}{4}=\frac{1}{2}.$$
Likewise we can check that $(r-p\sp\perp)^2=1/2$, hence $\|r-p\|=\|r-p\sp\perp\|=1/\sqrt{2}$, and the result follows.
\end{proof}

\begin{thm}
Let $M$ be a finite von Neumann algebra.
For any two homotopic projections $p,q$ we have $U_p\cap U_q \neq\emptyset$.
\end{thm}

\begin{proof}
Let $\pi$ be the projection onto $\mbox{Ker}\,(p+q-1)$.
By Lemma \ref{kerpq}, we know that $\pi$ is the projection onto $\mbox{Im}\,p\cap\mbox{Ker}\,q\oplus\mbox{Ker}\,p\cap\mbox{Im}\,q$.
Clearly $\pi$ belongs to the bicommutant $\{p,q\}''$ so that $\pi$ belongs to $P(M)$.
Note that $\pi$ also belongs to the commutant $\{p,q\}'$.\\
Let $p_0$ and $q_0$ denote $\pi p$ and $\pi q$, restrictions to the subspace $\mbox{Ker}\,(p+q-1)$.
They lie in the von Neumann algebra $\pi M\pi$.
Moreover, we have $p_0+q_0-\pi=0$, i.e. $q_0=p_0\sp\perp$.\\
Likewise, if we let $p_1$ and $q_1$ denote $\pi\sp\perp p$ and $\pi\sp\perp q$, we obtain two projections in the von Neumann algebra $\pi\sp\perp M\pi\sp\perp$.
By construction, we see that $p_1+q_1-\pi\sp\perp$ is now injective.\\
By Proposition \ref{genericcase}, we can find a projection $r_1$ in $P(\pi\sp\perp M\pi\sp\perp)$ such that $\|r_1-p_1\|<1$ and $\|r_1-q_1\|<1$.
In particular, by Lemma \ref{corokovarikbis} we see that $p_1$ and $q_1$ are homotopic in $\pi\sp\perp M\pi\sp\perp$.
By finiteness \cite[\rm{III}.1.3.8]{blackadarems}, it follows that $p_0$ and $q_0$ are equivalent in $M$.
Since the latter is a finite von Neumann algebra, it follows that $p_0$ and $q_0$ are homotopic in $M$.
In particular, there exists $u$ in $M$ such that $uu\sp*=p_0$ and $u\sp*u=q_0$.
Now without loss of generality, we can assume that $u=q_0up_0$.
So $u$ belongs to $\pi M\pi $, so that $p_0$ and $q_0$ are homotopic in $\pi M\pi $.\\
By Proposition \ref{pperpcase}, we can find a projection $r_0$ in $\pi M\pi $ such that $\|r_0-p_0\|<1$ and $\|r_0-q_0\|<1$.\\
Finally, it suffices to put $r:=r_0 + r_1$ to obtain the desired projection in $U_p\cap U_q$.
\end{proof}

\section{Infinite case}\label{infinitecase}

\begin{prop}\label{constructpq}
Let $M$ be a properly infinite von Neumann algebra.
Then there exist two distinct homotopic projections $p,q$ in $P(M)$ such that $p\leq q$.
\end{prop}

\begin{proof}
By \cite[\rm{III}.1.3.3]{blackadarems}, $M$ contains a unital copy of $O_\infty$.
So there is a sequence $(p_n)$ of projections in $P(M)$ such that 
$$1=\sum_{n\geq 0} p_n\qquad\mbox{with}\quad p_n\sim 1 \quad\forall n.$$
Then we can consider the three projections
$$ e_j:=\sum_{n\geq 0} p_{3n+j}\qquad j=0,1,2.$$
By additivity of the Murray-von Neumann equivalence relation \cite[\rm{III}.1.1.2]{blackadarems}, we see that 
$$e_0\sim e_1\sim e_2\sim e_0+e_1\sim e_1+e_2.$$
Now set for instance $p:=e_0$ and $q:=e_0+e_1$.
We have 
$$p\sim q\sim p\sp\perp \sim q\sp\perp.$$
It follows from Proposition 2.2 in \cite{rordamlms} that $p$ and $q$ are unitarily equivalent, hence homotopic.
The fact that they are distinct and that $p\leq q$ is obvious.
\end{proof}

\begin{thm}
Let $M$ be an infinite von Neumann algebra.
There exist two homotopic projections $p,q$ in $P(M)$ such that $U_p\cap U_q=\emptyset$.
\end{thm}

\begin{proof}
By Proposition \ref{constructpq}, we can take two distinct homotopic projections $p,q$ such that $p\leq q$.
We claim that $U_p\cap U_q=\emptyset$.
Assume for a contradiction that we can find a projection $r$ in $U_p\cap U_q$.
Then we can construct two idempotents $p_1$ and $p_2$ in $M$ by the formula of Lemma \ref{kovarik}:
$$p_1:=p(p+r-1)^{-2}r\quad\mbox{and}\quad p_2:=q(q+r-1)^{-2}r.$$
It follows from the remarks at the beginning of Section 3 in \cite{giol05} that each of the three segments $[p,p_1]$, $[p_1,p_2]$, and $[p_2,q]$ is contained in the set of idempotents of $M$.
Hence by Theorem 2.4 in \cite{giol05}, we have $p=q$, a contradiction.
\end{proof}

A simple modification of this proof allows to show that there exists a sequence $(p_k)$ of projections such that, for every $k\neq l$, $p_k$ and $p_l$ be homotopic and $U_p\cap U_l=\emptyset$.
For this it suffices to take $(e_{n,m})$ a sequence of pairwise homotopic and orthogonal projections such that 
$$1=\sum_{n,m} e_{n,m}.$$
Then if we set 
$$
p_k=\sum_{n=1}^k \sum_{m=1}^\infty e_{n,m}$$
we obtain such a sequence.


\end{document}